\documentclass[11pt]{amsart}
\usepackage[cp1251]{inputenc}
\usepackage{color}
\usepackage{amsthm,amssymb,amsmath,latexsym}
\usepackage{graphicx}
\usepackage[T2A]{fontenc}
\usepackage{graphs}
\usepackage{enumerate}
\usepackage{verbatim}

\theoremstyle{definition}
\newtheorem{defin}{Definition}[section]
\theoremstyle{remark}
\newtheorem{example}[defin]{Example}

\newtheorem{remar}[defin]{Remark}
\theoremstyle{plain}
\newtheorem{thm}[defin]{Theorem}
\newtheorem{prop}[defin]{Proposition}
\newtheorem{lemm}[defin]{Lemma}

\def\ov{\overline}

\newcommand{\inte}{\mathsf{int}}

\newcommand{\sq}{sequence}

\newcommand{\z}{\mathbb Z}
\newcommand{\na}{\mathbb N}

\newcommand{\tl}{topological}

\newcommand{\xt}{$(X,T)$}

\newcommand{\diam}{\mathsf{diam}}

\newcommand{\Per}{\mathsf{Per}}

\numberwithin{equation}{section}

\begin{document}

\title{Decisive Bratteli-Vershik models}

\date{}

\author{T. Downarowicz}
\address{Faculty of Mathematics and	Faculty of Fundamental Problems of Technology, Wroclaw University of Technology, Wroclaw, Poland}
\email{Tomasz.Downarowicz@pwr.edu.pl}

\author{O. Karpel}
\address{B. Verkin Institute for Low Temperature Physics and Engineering, Kharkiv, Ukraine\\
{\em Current address: Department of Dynamical Systems, Institute of Mathematics of Polish Academy of Sciences, Wroclaw, Poland}}
\email{helen.karpel@gmail.com}

\begin{abstract}
In this paper we focus on Bratteli-Vershik models of general compact zero-dimensional systems with the action of a homeomorphism. An ordered Bratteli diagram is called decisive if the corresponding Vershik map prolongs in a unique way to a homeomorphism of the whole path space of the Bratteli diagram. We prove that a compact invertible zero-dimensional system has a decisive Bratteli-Vershik model if and only if the set of aperiodic points is either dense, or its closure misses one periodic orbit.
\end{abstract}

\maketitle

\section{Introduction}
Bratteli-Vershik representations (BV-models for short) have been used to study mainly minimal Cantor systems, where they showed extremely useful as a tool allowing to describe the simplex of invariant measures and orbit equivalence classes (see e.g.~\cite{durand:2010}). It was proved in~\cite{herman_putnam_skau:1992} that every such system $(X,T)$ has a BV-model which determines it,
in the sense that the Vershik map is defined but at one point and prolongs continuously to the whole path-space, producing a system conjugate to $(X,T)$. Clearly, in such case the prolongation is \emph{unique}.

But the applicability of ordered Bratteli diagrams (with the action of the Vershik map) reaches much further. They may represent non-minimal Cantor dynamics (see \cite{downarowicz_survey}, \cite{BKM09},\cite{bezuglyi_karpel}), as well as Borel \cite{BDK06} and measurable dynamics (see \cite{vershik:1981}, \cite{vershik:1982}). By nature, they always produce zero-dimensional models, not necessarily minimal and not necessarily Cantor (i.e., they admit isolated points) and the Vershik map is not always uniformly continuous (see \cite{downarowicz_survey}). In this paper, we focus on Bratteli-Vershik models of general compact zero-dimensional systems with the action of a homeomorphism. There are several works devoted to establishing a class of such systems which fit the Bratteli-Vershik scheme best. However, a problem arises, which was not considered by other authors (see for instance \cite{medynets:2006}, \cite{shimomura:2016}): without assuming minimality (or at least \emph{essential minimality}, i.e., containing a unique minimal subset), the domain of the Vershik map usually misses a larger subset of the path-space and it is not clear how the map should be prolonged. Many authors so far were satisfied when a zero-dimensional system $(X,T)$ had a BV-model such that the Vershik map \emph{admitted}
a continuous prolongation conjugate to $(X,T)$. But a priori the Vershik map could also be prolonged to a different continuous map, producing a system not conjugate to $(X,T)$. In other words,
the same BV-model could serve for several mutually non-conjugate zero-dimensional systems.
%what larger subset?

The simplest example of this phenomenon is the Bratteli diagram in form of the binary tree. The path-space represents the Cantor set, but since every path is both maximal and minimal, the Vershik map is undefined everywhere and any continuous action ``fits'' to this model. What we require from a BV-model is exactly the opposite---it should \emph{determine} the system $(X,T)$ uniquely.

To make our point completely clear, we now give, after \cite{downarowicz_survey}, the definition of the notion which we consider crucial. All necessary terms will be provided in Preliminaries. For now it suffices to remember that a BV-model consists of a Bratteli diagram $B$ equipped with a partial order $<$ which induces a partly defined Vershik map $T_V$ on the path-space $X_B$.

\begin{defin}
We say that an ordered Bratteli diagram $(B,<)$ is \emph{decisive} if the Vershik map $T_V$ prolongs in \emph{a unique way} to a homeomorphism $\ov T_V$ of $X_B$. A zero-dimensional dynamical system $(X,T)$ will be called \emph{Bratteli-Vershikizable} if it is conjugate to $(X_B, \ov T_V)$ for a decisive ordered Bratteli diagram $(B,<)$.
\end{defin}
Roughly speaking a system is Bratteli-Vershikizable if it admits a BV-model which carries \emph{all the information} about the system. Clearly, all minimal Cantor systems are Bratteli-Vershikizable (they can be thought of as a prototype for this property).

The following proposition gives a criterion for decisiveness of an ordered Bratteli diagram (easy arguments for necessity and sufficiency are given in \cite{downarowicz_survey}).

\begin{prop}\label{decis_krit}

An ordered Bratteli diagram $(B,<)$ is decisive if and only if the following two conditions hold:

\begin{enumerate}[(A)]
	\item the Vershik map and its inverse are uniformly continuous on their domains, and
	\item the domains of the Vershik map and its inverse are either both dense in $X_B$ or their closures
	both miss one point (not necessarily the same).
\end{enumerate}
\end{prop}

Since uniform continuity is automatic in BV-models of continuous dynamical systems, it is the density of the domain of the Vershik map that becomes the crucial ingredient. Based on the approach in~\cite{medynets:2006} one might have the impression that this density is implied by the condition that the heights of the Kakutani-Rokhlin towers, associated with the diagram, grow to infinity. Medynets shows that such BV-models can be obtained for any aperiodic zero-dimensional system. While the last result is an almost immediate consequence of the Krieger's Marker Lemma, it is in fact insufficient for density of the domain of the Vershik map. An appropriate example is provided in Section \ref{examples}. Nonetheless, it is true that every aperiodic zero-dimensional system admits a decisive BV-model. This fact is proved in our recent paper \cite{downarowicz_survey}, but it requires a little more work than just a straightforward application of the marker lemma.

But aperiodicity is not necessary for the existence of a decisive BV-model and the goal of this paper is
to give full characterization of the class of Bratteli-Vershikizable systems. Moreover, we also examine the notion of weak decisiveness (when the prolongations of the Vershik maps are numerous, but all yield mutually conjugate systems). We show that systems admitting weakly decisive BV-models also admit decisive BV-models. Finally, in Section \ref{fullshift}, we describe a decisive BV-model of the full shift on two symbols, which turns out surprisingly irregular.

\medskip
Let us also mention that decisive ordered Bratteli diagrams have been in fact studied, for instance in~\cite{BKY14, BY}, but from a different perspective (and without using the notion of decisiveness). The authors start with a fixed Bratteli diagram $B$ and search for orderings $<$ such that the corresponding Vershik map can be uniquely prolonged to a homeomorphism of the whole path-space. Such orderings are called perfect (in our notation, these are precisely the decisive orderings). In our approach the starting point is always the abstract system $(X,T)$ and then we search for both the diagram and a perfect ordering on it.

\section{Preliminaries}
We begin with some elementary general facts concerning zero-dimensional systems.

\begin{defin}
Let $\Lambda_1,\Lambda_2,\dots$ be finite alphabets (the cardinalities need not be bounded).
By an \emph{array system} we mean any closed, shift-invariant subset of the Cartesian product
$\prod_k \Lambda_k^\z$. Each element of the array system can be pictured as an array $x=[x_{k,n}]_{k\in\na,n\in\z}$, such that each $x_{k,n}$ belongs to $\Lambda_k$.
Speaking about an array we will refer to the indices $k$ and $n$ as \emph{vertical} and
\emph{horizontal} coordinates (positions), respectively. On this space we consider the action of the horizontal shift $Tx =[x_{k,n+1}]_{k\in\na,n\in\z}$.
\end{defin}

Clearly, every array system is zero-dimensional. The converse also holds, which is an elementary fact
(see e.g. \cite{downarowicz_survey}):

%\begin{defin}
%Let $(X_k,T_k)$ be countable \sq\ of \tl\ \ds s such that $(X_k,T_k)$ is a topological factor of $(X_{k+1},T_{k+1})$, for each $k\ge 1$. The corresponding (surjective) factor maps $\psi_k:X_{k+1}\to X_k$ are called the \emph{bonding maps}. The \emph{inverse limit} $(X,T)$ of these systems is the subsystem of the Cartesian product $\prod_{k=1}^\infty(X_k,T_k)$ (with the product action) defined by the rule
%$$
%(x_k)_{k\ge1}\in\prod_kX_k \text{ \ belongs to the inverse limit $X$ if and only if \ } \forall_k x_k = \psi_k(x_{k+1}).
%$$
%The inverse limit is denoted by $(X,T)=\overset{\longleftarrow}\lim_k(X_k,T_k)$.
%\end{defin}

\begin{thm}
Every zero-dimensional system \xt\ is conjugate to an array system.
\end{thm}
\medskip

We now recall briefly the notion of an ordered Bratteli diagram and the associated Vershik map. For more details see \cite{herman_putnam_skau:1992}.
\smallskip

A Bratteli diagram is a graph $B=(V,E)$ whose set of vertices $V$ is organized into countably many disjoint finite subsets $V_0$, $V_1$, \dots\ called {\it levels}. The zero level $V_0$ is a singleton $\{v_0\}$. The set of edges $E$ of the diagram is organized into countably many disjoint finite sets $E_1, E_2,\dots$. Every edge $e\in E_k$ connects a \emph{source} $s=s(e)\in V_k$ with some \emph{target} $t=t(e)\in V_{k-1}$. Each vertex is a target of at least one edge and every vertex of each level $k>0$ is also a source of at least one edge. Multiple edges connecting the same pair of vertices are admitted. By a \emph{path} we understand a finite (or infinite) \sq\ of edges $p=(e_1,e_2,\dots,e_l)$ (or $p=(e_1,e_2,\dots)$) such that $t(e_{k+1})=s(e_k)$ for every $k=1,\dots,l-1$ (or $i=1,2,\dots$). Then the target of $e_1$ will be referred to as the target of the path and (only for finite paths) the source of $e_l$ will be referred to as the source of the path.

\begin{defin}
Given a Bratteli diagram $B$, we define the \emph{path space} $X_B$ as the set of all infinite paths with target $v_0$. We endow $X_B$ with the topology inherited from the product space $\prod_k E_k$, where each $E_k$ is considered discrete. Clearly, $X_B$ is compact, metric and zero-dimensional.
\end{defin}

\begin{defin}
By an \emph{ordered Bratteli diagram} $(B,<)$ we shall mean a Bratteli diagram $B$ with a specific partial order. For each vertex $v\in V_k$, where $k>0$, all edges $e$ with $s(e)=v$ are ordered linearly
(i.e., enumerated as $\{e_1,e_2,\dots,e_{n(v)}\}$). Edges with different sources are incomparable.
\end{defin}

The above order allows one to introduce a partial order among finite and infinite paths. Two finite paths are comparable if they have a common source and the same length. For such paths we can apply the inverse lexicographical order: a path $p=(e_1,e_2,\dots,e_l)$ precedes $p'=(e'_1,e'_2,\dots,e'_l)$ if there exists
an index $1\le i\le l$ such that $e_j = e'_j$ for all $j>i$ (then $s(e_i)=s(e'_i)$) and $e_i<e'_i$ (we admit $i=l$; then the first condition is fulfilled trivially). Two infinite paths are comparable if they have targets in the same level and they are \emph{cofinal}, i.e., they agree from some place downward. In such case we apply to them the same rule as described above. It is easy to see that the relation of being cofinal for two paths is an equivalence relation.

\begin{defin} A finite or infinite path is called \emph{maximal} (\emph{minimal}) if it has no successor (predecessor).
\end{defin}
By compactness, one can show that at least one maximal and one minimal path in $X_B$ always exist. The sets of maximal paths and of minimal paths are closed and we will denote them by $X_{\max}$ and $X_{\min}$, respectively.

\begin{defin} On the path space $X_B$ of an ordered Bratteli diagram $(B,<)$, there is a natural, partially defined transformation $T_V$, called the \emph{Vershik map}. It is defined on the set of all but maximal paths and it sends every such path to its successor. The range of the map is the set of all but minimal paths. The Vershik map is a homeomorphism between its domain and range.
\end{defin}

Given a zero-dimensional system \xt\ we will say that an ordered Bratteli diagram $(B,<)$ is a \emph{model} for \xt\ if the Vershik map $T_V$ admits a continuous prolongation $\ov T_V$ to the whole path-space $X_B$ such that the systems $(X_B,\ov T_V)$ and \xt\ are \tl ly conjugate.

In view of what was already explained in the Introduction, for an ordered Bratteli diagram to be decisive, the interiors of the set of maximal paths and that of minimal paths must either both be empty or both be singletons. For a moment we will focus on the first case. It is fairly obvious that in
any BV-model of a compact dynamical system, any periodic orbit contains at least one maximal and one minimal path. Thus, in order to fulfill the condition that the sets of maximal and minimal paths have empty interiors (recall that these sets are closed, so in fact we need them to be nowhere dense), we must be at least able to find a nowhere dense selector from periodic orbits. This is managed in the elementary lemma below.

\begin{lemm}\label{periodic_cross_section}
Let $(X,T)$ be a zero-dimensional dynamical system. A nowhere dense set intersecting every periodic orbit exists if and only if the set of aperiodic points is dense in $X$.
\end{lemm}

\begin{proof}
Suppose that the set of aperiodic points is dense. Then there exists a dense sequence $\{x_k\}_{k=1}^\infty$ of aperiodic points. Let $\{p_k\}_{k=1}^\infty$ be a strictly increasing sequence of positive integers such that $\sum_{k=1}^\infty \frac{1}{p_k} < 1$. Let $U_k$ be a clopen neighbourhood of $x_k$ such that the sets $U_k$, $T U_k, \ldots, T^{p_k-1}U_k$ are pairwise disjoint. Denote by $\mathsf{Per}(X,T)$ the set of all periodic points of $(X,T)$. Let $A = \mathsf{Per}(X,T) \setminus \bigcup_{k=1}^{\infty} U_k$. Then $A$ is nowhere dense (its closure misses the dense \sq\ $\{x_k\}_{k=1}^{\infty}$). Let $x$ be any periodic point and denote its minimal period by $n$. Then the set $U_k$ contains at most the fraction $\frac1{p_k}$ out of $n$ elements of the orbit of $x$. Since $\sum_{k=1}^\infty \frac{1}{p_k} < 1$, the set $A$ intersects the orbit of $x$.

Conversely, suppose that $A$ is a nowhere dense set which intersects every periodic orbit. Then $\mathsf{Per}(X,T)\subset\bigcup_{n = 0}^{\infty} T^n(A)$ which is a first category set, hence, by the Baire Theorem, its complement is dense.
\end{proof}

\section{Formulation of the main result}
The preceding lemma has led us to the condition that the set of aperiodic points is dense. It turns out that this condition is precisely all we need to claim the system to be Bratteli-Vershikizable. As a matter of fact, we may also admit in the system one isolated periodic orbit. Below is the formulation of our main result---the characterization of Bratteli-Vershikizable systems.

\begin{thm}\label{main}
A (compact, invertible) zero-dimensional system $(X,T)$ is Bratteli-Vershikizable if and only if the set of aperiodic points is dense, or its closure misses one periodic orbit.
\end{thm}
The proof is given after a preparatory section containing the exposition of the main tool---the markers.

\section{Systems of markers} Throughout we fix a zero-dimensional system \xt. We denote by $\Per_{[1,n-1]}$ the set of all periodic points of \xt\ with minimal periods smaller than $n$. Our starting point is the Krieger's Marker Lemma (see \cite{Bo83}), which we will use in the following version:

\begin{thm}[Krieger's Marker Lemma]\label{kriegerg}
For every $\epsilon>0$ and every natural $n$ there exists a clopen set $F=F(n,\epsilon)$ such that:
\begin{enumerate}[(1)]
	\item no orbit visits $F$ twice in $n$ steps  (i.e., $F,TF,\dots,T^{(n-1)}F$ are pairwise disjoint;
	we will say that $F$ is \emph{$n$-separated}),
	\item $\bigcup_{i=-n}^n T^iF\supset X\setminus (\Per_{[1,n-1]})^\epsilon$,
\end{enumerate}
where $A^\epsilon$ denotes the $\epsilon$-neighborhood of a set $A$.
\end{thm}
In zero-dimensional systems given in the array representation, the times of visits of a point (array) $x$
in the marker set $F$ can be conveniently pictured as additional symbols (which we will call \emph{markers}) in form of short vertical bars inserted in a selected row of $x$ (we can assign the row numbers to the marker sets as we wish), by the following rule: whenever $T^nx\in F$ then we put a bar
(in the selected row) between the symbols at the horizontal positions $n$ and $n+1$. Because the marker sets are clopen, adding the markers produces a topologically conjugate representation of the system. In fact, we will do more than that: we choose a fast increasing \sq\ of positive integers $\{n_k\}_{k=1}^\infty$ and a decreasing to zero \sq\ of parameters $\{\epsilon_k\}_{k=1}^\infty$, and for each $k$ we call the set $F(n_k,\epsilon_k)$ the \emph{$k$th marker set}, and we denote it shortly by $F_k$. Also, we agree that the markers corresponding to the visits in $F_k$ (which we will call $k$-markers) will be put in row number $k$. When this is done for every $k\ge 1$, we obtain a conjugate \emph{array representation of \xt\ with markers} distributed in every row of every array. This is in fact a ``usual'' array representation of \xt\ which uses the enlarged alphabets $\Lambda_k^* = \Lambda_k\times\{\emptyset, |\} = \{a,a|: a\in\Lambda_k\}$.

In this setup, the Krieger's Marker Lemma has the following interpretation:
\begin{enumerate}[(1)]
	\item the markers in every row $k$ of every array appear with gaps at least $n_k$,
	\item arrays sufficiently distant from periodic points with periods less than $n_k$ have, in row $k$ and
	around the horizontal coordinate $0$,	markers appearing with gaps bounded by $2n_k+1$,
	\item in periodic arrays or arrays close to periodic ones markers in further rows $k$ may appear with
	gaps larger than $2n_k+1$ (even infinite, i.e., the markers may be missing on either side).
\end{enumerate}

The markers can be easily manipulated (shifted, added, removed, copied from one row to another, etc.). Every such manipulation translates to (usually complicated) set operations on the marker sets, but the array representation enables one to forget these complications. In order to keep our system conjugate,
we only need to make sure that our marker manipulations are
\begin{itemize}
	\item shift-equivariant, and
	\item depend locally on a bounded area in the array only (this is continuity),
\end{itemize}
(invertibility is automatic, as removing the markers is always a continuous procedure).
The lemma below connects the assumption on dense aperiodic points with the distribution of markers, and it
is the key ingredient of the proof of the main theorem.

\begin{lemm}\label{marker_lemma}
Let $(X,T)$ be a zero-dimensional dynamical system such that the set of aperiodic points is dense in $X$.
Then $(X,T)$ admits an array representation with markers\footnote{This markered array representation is different from a standard one introduced in~\cite{downarowicz_survey} for aperiodic systems.} such that the following restrictions hold, for every $k\ge 1$:

\begin{enumerate}[(1)]
	\item the markers in row $k+1$ are allowed only at horizontal positions of the markers in row $k$,
	\item the gaps between markers in row $k$ are bounded from above,
	\item the set of arrays which have a marker of infinite order is of first category.
\end{enumerate}
\end{lemm}
Before the proof let us explain that the first condition means simply that the corresponding marker sets are nested (i.e., $F_{k+1}\subset F_k$ for each $k$). In such case, a \emph{marker of infinite order} (referred to in the last condition) occurs at a horizontal position $n$ of some $x$ when $T^nx\in F_\infty := \bigcap_k F_k$ (then the markers form a vertical line extending through all rows). Thus the set of arrays which have a marker of infinite order equals $\bigcup_{n\in\z}T^nF_\infty$. Since $F_\infty$ is closed, to satisfy the last condition we only need to arrange that $F_\infty$ has empty interior (i.e., dense complement).

\begin{proof}
Fix some dense in $X$ sequence $\{x_k\}_{k=1}^\infty$ of aperiodic points. Let $\{n_k\}_{k=1}^\infty$ be a strictly increasing sequence of positive integers such that $n_{k+1}\ge \sum_{i=1}^k n_i+2$, for each $k\ge 1$. Also fix a decreasing to zero \sq\ of positive numbers $\{\varepsilon_k\}_{k=1}^{\infty}$ such that $x_k$ does not belong to $(\mathsf{Per}_{[1,n_k-1]})^{\varepsilon_k}$. Let $\widetilde{F}_k=\widetilde{F}(n_k,\epsilon_k)$ be the corresponding Krieger's marker set. In particular, the union $\bigcup_{i = -n_k}^{n_k}T^i (\widetilde{F}_k)$ contains $x_k$. We pass to the array representation of \xt\ with the markers corresponding to the sets $\widetilde{F}_k$. Now each $x_k$ has a marker in row $k$ at a position $i_k\in[-n_k,n_k]$, and the following marker in this row is at least $n_k$ positions further to the right. Let us now shift all the markers in row $k$ in all arrays by $i_k-1$ positions to the left (if $i_k\le 0$ we actually shift them to the right). As a result, $x_k$ has in row $k$ a marker precisely at the coordinate $-1$ and no markers in $[0,n_k-2]$. Notice that such uniform shifting (by the same vector) of all markers in row $k$ in all arrays is a continuous and shift-equivariant algorithm.

We will now put more markers in row $k$ (using another continuous and shift-equivariant algorithm). This will reduce the gap lengths so they become bounded. In every array we call the union of the intervals of length $n_k-1$ lying directly to the right of every marker ``the forbidden zone'' (note that there are no markers in the forbidden zone). Then we put new markers at every position outside the forbidden zone. After this step the markers in row $k$ appear with gaps bounded by $n_k$. Also note that in $x_k$ we still have no markers in $[0,n_k-2]$ in row $k$, because this interval was part of the forbidden zone.

Finally we apply yet another shift-equivariant and continuous algorithm, which we call ``upward adjustment''. It turns the marker sets into a nested \sq\ (i.e., it makes the markers fulfill the condition (1) of the assertion of the lemma). We proceed inductively. In step 1 we do not move the markers in row $1$. In step $k+1$ we move each marker in row $k+1$ to the left, until its horizontal position matches that of a marker in row $k$. Note that the movement is by at most $\sum_{i=1}^k n_i$, i.e., by at most $n_{k+1}-2$. It will often happen that several markers in row $k+1$ are moved to a common position (``glued together''), but this does not bother us. After this is done for every $k$, the construction is finished.

Now Condition (1) is clearly fulfilled. Condition (2) also holds; the upward adjustment can increase the gap sizes, but only by a bounded amount. For (3) observe that $x_k$ still has no marker at the coordinate $0$ of row $k$, i.e., $x_k\notin {F}_k$ and thus $x_k\notin {F}_\infty$. Since this is true for every $k$, the set ${F}_\infty$ has a dense complement, which ends the proof.
\end{proof}

\section{Proof of the main theorem}
\begin{proof}[Proof of Theorem \ref{main}]
Suppose the closure of the aperiodic points misses more than one periodic orbit. Then any Borel set intersecting every periodic orbit has an interior consisting of more than one point. This applies to the set of maximal paths of any BV-model, thus condition (B) given in the Introduction is not satisfied for any such model, and the system is not Bratteli-Vershikizable.

We pass to the proof of the inverse implication. Consider a system \xt\ such that $X$ consists of two parts: $X'$ equal to the closure of all aperiodic points and perhaps one isolated periodic orbit.
Suppose we have a BV-model for $X'$. It can be enhanced to a BV-model of $X$ by adding a separate diagram for the missing periodic orbit, having one maximal and one minimal paths. Note that both these paths are isolated points in the path-space model of $X$. Now suppose that the BV-model for $X'$ is not only decisive, but also that the sets of maximal and minimal paths both have empty interiors. In particular, they contain no isolated points. Then any prolongation of the Vershik map to the enhanced diagram must send the unique isolated maximal path to the unique isolated minimal path, closing the periodic orbit correctly. Thus, the enhanced diagram is decisive. Summarizing, it suffices to deal with systems \xt\ in which aperiodic points lie densely, and for every such system prove the existence of a decisive BV-model in which the set of maximal paths has empty interior (then the same will hold automatically for the set of minimal paths).

So, consider a system \xt\ with a dense set of aperiodic points. We apply Lemma~\ref{marker_lemma} to obtain
an array representation with markers satisfying (1)-(3). From this, we shall build an ordered Bratteli diagram. Let us introduce some terminology: By a $k$-block (appearing in $X$) we shall mean any block
of symbols from $\Lambda_k$ appearing between two neighboring markers in row $k$ of some array $x\in X$.
It is important that we ignore the position of the $k$-block along the horizontal axis, i.e., we think of a $k$-block as an element of the Cartesian power $\Lambda_k^l$ where $l$ it the block's length. We also include in the $k$-block both markers that embrace it; however, when concatenating $k$-blocks we must remember to ``glue'' the markers meeting at the contact places. A $k$-rectangle is the rectangular block of symbols that can be seen in rows $1$ through $k$ directly above a $k$-block (again, we ignore its horizontal position). Since the marker sets are nested, every $(k\!+\!1)$-rectangle $R$ is a concatenation of some number of $k$-rectangles $R^{(1)}\dots R^{(q)}$ with a $(k\!+\!1)$-block $B$ added in row $k\!+\!1$. Symbolically, we write this as
$$
R = \left[\begin{matrix} R^{(1)}R^{(2)}\dots R^{(q)}\\ B\end{matrix}\right].
$$
We need to enlarge slightly the $k$-rectangles; by a $k$-trapezoid we shall understand a configuration (appearing in some array $x\in X$) which consists of a $k$-rectangle (which we call the \textit{core} of the $k$-trapezoid) enlarged in rows $1$ through $k\!-\!1$ by two $(k\!-\!1)$-rectangles (one on each side), then, in rows $1$ through $k\!-\!2$ by two more $(k\!-\!2)$-rectangles (one on each side), etc. The figure below shows a $3$-trapezoid.

\begin{center}
\includegraphics[width=12cm]{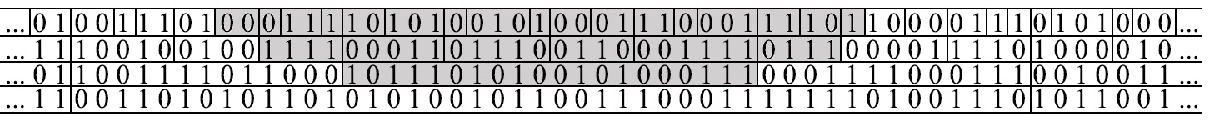}
\end{center}

Of course, every $k$-rectangle may be the core of several $k$-trapezoids differing in the added rectangles.
Notice that while each $(k\!+\!1)$-rectangle $R$ has in its top $k$ rows a concatenation of $k$-rectangles, say $R_1,R_2,\dots,R_q$, a $(k\!+\!1)$-trapezoid $S$ which extends $R$ has in its top $k$ rows an ``overlapping concatenation'' of $q\!+\!2$ $k$-trapezoids, say $S_0,S_1,\dots,S_q,S_{q+1}$.
We will call the $k$-trapezoids $S_1,\dots,S_q$ \emph{internal}, while $S_0$ and $S_{q+1}$ will be called \emph{external}. The internal $k$-trapezoids extend the $k$-rectangles included in $R$, the external ones do not.

We define the vertex sets of the Bratteli diagram as follows: level $0$ consists of one vertex $v_0$,
each level $k\ge 1$ consists of all possible $k$-trapezoids occurring in $X$. Since the gaps between markers are bounded, each level set is finite. The edges with source in a $(k+1)$-trapezoid $R$ connect it with all its internal $k$-trapezoids and the order corresponds to the natural order as they appear in the ``overlapping concatenation'' (if $S_i=S_j$ for some $i\neq j$, we obtain multiple edges with the same source and target).

Every infinite path with target at $v_0$ of this diagram corresponds to a nested \sq\ of congruent (i.e., agreeing on the common area) and growing $k$-trapezoids, with a determined position of the horizontal coordinate zero always inside the core. Together they define an entire array $x\in X$ with a determined position of the horizontal coordinate zero. It is elementary to see that this correspondence is a homeomorphism between the path space $X_B$ and the array representation $X$. We remark that this might not work with $k$-rectangles; a path of $k$-rectangles could define only a left or right half-array or even just a vertical strip of bounded width. This is the reason why we use trapezoids. Further, it is easy to see that a path is maximal if and only if the respective array has a marker of infinite order at the horizontal coordinate zero and that the Vershik map on non-maximal paths corresponds to the left shift of the respective arrays. The former fact (in view of property (2) in Lemma \ref{marker_lemma}) implies that the set of maximal paths has empty interior, the latter gives uniform continuity of the Vershik map and coincidence of its unique prolongation with the shift. This ends the proof.
\end{proof}

\begin{remar} According to Theorem~\ref{main}, in any case of \xt, the subsystem obtained as the closure of the set of aperiodic points always possesses a decisive Bratteli--Vershik model $(B,<)$. Then one can extend the ordered diagram so that some prolongation of the Vershik map models the entire system. This follows directly from the proof above. Indeed, pick the sequence of aperiodic points $\{x_k\}$ which is dense in the closure of the set of all aperiodic points. Then, by the proof of Theorem~\ref{main}, we can build a BV-model of \xt\ such that the interior of $X_{\max}$ misses the closure of the set of aperiodic points. Then this model is decisive on this closure. In other words, one can say that \emph{non-decisiveness concerns only the interior of the set of periodic points}, i.e., there always exists a BV-model such that all prolongations differ only on this set.
\end{remar}

\section{Weak decisiveness}
Some ordered diagrams, although not decisive, have the property that every prolongation of the Vershik map produces a conjugate model of the same system, which is made precise in the definition below. In a sense, these diagrams still completely determine the dynamics of the system.

\begin{defin}
We call an ordered Bratteli diagram \emph{weakly decisive} if the Vershik map can be prolonged to a homeomorphism of the whole path space and all such prolongations are topologically conjugate.
\end{defin}

An example of a non-decisive, weakly decisive Bratteli diagram is given in \cite[Example 6.15]{downarowicz_survey} and copied below as Example~\ref{medynets}. In this example, weak decisiveness follows from the fact that the images of $X_{\min}$ shrink to zero in diameter, as stated in the following proposition:

\begin{prop}\label{weak dec}
Let $(B,<)$ be an ordered Bratteli diagram such that the Vershik map $T_V$ can be prolonged to a homeomorphism $\overline T_V$ of the whole path space $X_B$, the sets $X_{\max}$ and $T_V^n(X_{\min})$ for $n\ge 0$ are pairwise disjoint and $\lim_n\diam(T_V^n(X_{\min}))=0$. Then $(B,<)$ is weakly decisive.
\end{prop}

\begin{remar}
In fact it suffices that the sets $\inte(X_{\max})$ and $T_V^n(\inte(X_{\min}))$ are disjoint and have diameters shrinking to zero. Since the proof requires more details, while the above, simpler formulation works in Example~\ref{medynets}, we have given up the weaker assumption.
\end{remar}

\begin{proof}
Let $T_1$ and $T_2$ be two different homeomorphisms which are both prolongations of $T_V$, i.e.
$T_1(x) = T_2(x)=T_V(x)$ for every $x \notin \inte(X_{\max})$. Clearly, $T_1(X_{\max})=T_2(X_{\max})=X_{\min}$. The disjointness condition implies that for any $n\ge 1$, we also have $T_1^n(X_{\max})=T_2^n(X_{\max})$. Denote
$$
U=\bigcup_{n\ge 0}T_1^n(X_{\max}) = \bigcup_{n\ge 0}T_2^n(X_{\max}) = X_{\max}\cup\bigcup_{n\ge 0}T_V^n(X_{\min}).
$$
For each $x\in U$ let $n_x\ge 0$ be the unique index such that $x\in T_1^{n_x}(X_{\max})$
(equivalently, $x\in T_2^{n_x}(X_{\max})$). Define
$$
h(x) = \begin{cases}
T_2^{n_x}(T_1^{-n_x}(x));& \text{ if $x\in U$}\\
x;& \text{ otherwise}.
\end{cases}
$$
It is now clear that $h$ is a bijection (the inverse is defined by the same formula with $T_1$ and $T_2$ exchanged) and that it is continuous on $U$. For continuity of $h$, we only need to consider points $x$ which are limits of \sq s of the form $\{x_k\}_{k\ge 1}$ such that $x_k\in T_1^{n_k}(X_{\max})$ with $n_k\to\infty$. Since both $h(x_k)$ and $x_k$ are in $T_1^{n_k}(X_{\max})=T_V^{n_k-1}(X_{\min})$, the distance between these two points tends to zero. It remains to show that $h(x)=x$. This is obviously true if $x$ does not belong to the union $U$. If $x$ belongs to some $T_1^{n_x}(X_{\max})$ then it belongs to the boundary of this set, which implies that $T_1^{-n_x}(x)$ belongs to the boundary of $X_{\max}$, while on this latter boundary $T_2=T_1$, implying $h(x)=x$.

Finally, we need to check whether $h(T_1(x)) = T_2(h(x))$ for all $x \in X_B$.
Indeed, if $x\in U$ then $n_{T_1(x)} = n_x+1$ and thus
$$
h(T_1(x)) = T_2^{n_x+1}(T_1^{-n_x-1}(T_1(x))) = T_2^{n_x+1}(T_1^{-n_x}(x)).
$$
On the other hand,
$$
T_2(h(x)) = T_2(T_2^{n_x}(T_1^{-n_x}(x))) = T_2^{n_x+1}(T_1^{-n_x}(x)).
$$
In case when $x\notin U$ we have $h(T_1(x)) = T_1(x) \text{ while } T_2(h(x)) = T_2(x)$
(notice that $T_1(x)$ may belong to $X_{\max}$, but on $X_{\max}$ $h$ is also the identity function).
Finally, for such an $x$, $T_1(x) = T_2(x)$, so the desired equality also holds.
\end{proof}

A question arises: are there zero-dimensional systems which do not possess a decisive BV-model (are not Bratteli-Vershikizable) yet admit a weakly decisive BV-model. These systems could also be considered
Bratteli-Vershikizable (in a weaker sense). We will show that such systems do not exist, i.e.,
any system either admits a decisive BV-model or each of its BV-models admits non-conjugate prolongations.

\begin{thm}\label{ble}
Let \xt\ be a zero-dimensional system which is not Bratteli-Vershikizable. Then, the Vershik map $T_V$ on any  BV-model of \xt\ admits a prolongation $(X_B,\widetilde T_V)$ to the path space, not conjugate to \xt.
\end{thm}

\begin{proof} Throughout the proof we fix some BV-model of \xt. By Theorem \ref{main}, the set of periodic points of \xt\ has a non-empty interior larger than just one isolated orbit. We begin with a lemma:

\begin{lemm}\label{pom}
Let $P_p$ be an open set consisting of $p$-periodic orbits. Then in the BV-model of \xt, there exists a relatively clopen in $P_p$ (hence open) set $M_p$ selecting exactly one maximal path from each periodic orbit in $P_p$.
\end{lemm}
\begin{proof} First we show that there exists an open (in $X$) set $U$ selecting exactly one point from each periodic orbit in $P_p$. The construction resembles that of a $p$-marker set. Each point $x$ in $P_p$ has a clopen neighborhood $U_x$ such that $U_x,\ T(U_x),\ \dots,\ T^{p-1}(U_x)$ are pairwise disjoint. The cover $\{U_x:x\in P_p\}$ has a countable subcover which we denote by $\{U_n:n\ge 1\}$. We let
$$
U = \bigcup_{n\ge 1} \Bigl(U_n\setminus\bigl(\bigcup_{1\le m<n,\ k=0,\dots,p-1}T^k(U_m)\bigr)\Bigr).
$$
We skip the straightforward verification that $U$ is the desired set. Since $P_p$ is the disjoint union of $U\cap P_p,\ T(U\cap P_p),\ \dots,\ T^{p-1}(U\cap P_p)$, the set $U\cap P_p$ is relatively clopen in $P_p$. Now recall that in the BV-model, each periodic orbit contains at least one maximal path and the set of maximal paths is closed. Let $P_p^{\max}$ be the (relatively closed in $P_p$) set of maximal paths in $P_p$. For each $i=0,1,\dots,p-1$ define $P_p^{i} = P_p^{\max}\cap T^i(U)$. These sets are relatively clopen in $P_p^{\max}$ and their union is $P_p^{\max}$. We can now define
$$
M_p=\bigcup_{i=0}^{p-1}\Bigl(P_p^{i}\setminus\bigl(\bigcup_{0\le j<i,\ k=0,\dots,p-1}T^k(P_p^{j})\bigr)\Bigr).
$$
Again, we skip the straightforward verification that $M_p$ is a selector from orbits contained in $P_p$, consisting of maximal paths. Since this set is obviously clopen in $P_p^{\max}$, it is closed in $P_p$, and since the disjoint union $M_p\cup T(M_p),\cup\cdots\cup T^{p-1}(M_p)$ equals $P_p$, $M_p$ is in fact clopen in~$P_p$.
\end{proof}

We continue with the main proof in which we will consider two main cases.
\begin{enumerate}
	\item[(I)] There are at least two isolated periodic orbits.
	\begin{enumerate}
	\item For some period $p$ the collection of all isolated $p$-periodic orbits is finite non-empty.
	\item The opposite.
  \end{enumerate}
	\item[(II)] There is at most one isolated periodic orbit.
\end{enumerate}

In case (Ia) it is very easy to define a prolongation $\widetilde T_V$ of the Vershik map, such that all isolated orbits of period $p$ are joined into one isolated periodic orbit of a larger period (in case the isolated orbit of period $p$ is unique we can combine it with an isolated orbit of a different period, which exists by the assumption (I)). It this manner the period $p$ no longer appears in isolated orbits for $\widetilde T_V$, which makes it not conjugate to $T$.

In case (Ib), let $p$ be some period occurring in the isolated periodic orbits. Let $P_p$ denote the union of all isolated $p$-periodic orbits. Clearly, $P_p$ is open. By assumption, $P_p$ is infinite, on the other hand it is at most countable. Now we refer to the open set $M_p$ of Lemma \ref{pom}, which, in this case is infinite and countable. The set of accumulation points of $M_p$ is non-empty and disjoint from $M_p$, so it equals the boundary $\partial M_p$. We cover $\partial M_p$ by finitely many disjoint clopen sets $D^{(1)}_i$, each of diameter less than some $\delta_1$, and so that the set $M_p\setminus\bigcup_i D^{(1)}_i$ contains more than one point. Since the last set is finite and all its members are maximal in some isolated $p$-periodic orbits, the map $T_V$ can be prolonged onto these points in such a way that all these orbits are glued into one periodic orbit of period larger than $p$. Next, denote by $M^{(1)}_p$ the set $M_p$ with the finitely many isolated points belonging to the glued orbit removed. We cover $\partial M_p$ again by finitely many disjoint clopen sets $D^{(2)}_j$, each of diameter less than some $\delta_2$ (much smaller than $\delta_1$) and contained in some set $D^{(1)}_{i(j)}$, and so that for each $i$, the difference $D^{(1)}_i\setminus\bigcup_j D^{(2)}_j$ contains more than one point (there are finitely many such points and they belong to $M^{(1)}_p$). Again, the orbits of these points can be glued into one orbit. Notice that on each set $D^{(1)}_i$ such prolongation differs from $T$ by less than some small $\varepsilon_1$ (associated to $\delta_1$ via the modulus of uniform continuity of~$T$).

We continue in the same manner, choosing the numbers $\delta_k$ tending to zero. The resulting new prolongation $\widetilde T_V$ will be continuous: the only points at which continuity needs a verification are the elements of $\partial M_p$. But the distance between $\widetilde T_V$ and the continuous map $T$ tends to zero as we approach this set, so the continuity follows. Now, $\widetilde T_V$ has no isolated periodic orbits of period $p$ at all, and thus it cannot be conjugate to $T$.

Case (II): We remove the unique isolated periodic orbit (if it exists) and throughout the remainder of this proof treat it as non-existing. The interior of the set of periodic points remains non-empty and now has no isolated points. By the Baire Theorem, there exists a period $p$ such that $\inte(\Per_p)$, in this part of the proof denoted as $P_p$, is non-empty. As before, we refer to the open selector set $M_p$. Since $P_p$ is open, it cannot contain relatively isolated points, and since $M_p$ is relatively open in $P_p$ (hence open), it has no relatively isolated points either. Since $X$ is zero-dimensional, it follows that $M_p$ is a countable union of clopen sets $C_n$ with no isolated points. All such sets are homeomorphic to the Cantor set and hence to each other. By further dividing, we can arrange that $\diam(C_n)\to 0$. From here we repeat the proof as in case (Ib) with the only difference that instead of isolated points we have isolated Cantor sets shrinking in diameter to zero as they approach points in $\partial M_p$. As a result, we will construct a prolongation $\widetilde T_V$ for which the set of $p$-periodic points (after removing the unique isolated periodic orbit) is empty. Thus, $\widetilde T_V$ is not conjugate to $T$.
\end{proof}

\begin{comment}
Below we give an example of a non-weakly decisive Bratteli diagram with the aperiodic cofinal equivalence relation. It is obtained by taking a disjoint union of two diagrams from Example 6.15.

\end{comment}

\section{Examples}\label{examples}

In this section we give examples of decisive and non-decisive Bratteli diagrams.

\begin{example}
The first figure below shows a decisive Bratteli-Vershik system with a dense set of aperiodic points and a dense countable set of periodic points. The ordering is not shown, but on this diagram there is practically just one order. It is not hard to check that the Vershik map can be prolonged to a homeomorphism of the whole path space $X_B$ and that the set $X_{\max}$ has empty interior. Hence, by Proposition~\ref{decis_krit}, the Bratteli diagram is decisive. After the prolongation, there is one fixed point and, for each $k\ge 1$, infinitely countably many periodic points of period $2^k$. The diagram models a Cantor dynamical system with uncountably infinitely many closed minimal subsystems, each of the subsystems is either a periodic orbit or a dyadic odometer (a path represents a periodic point if, from some place downward, it passes only via the ``single'' edges, otherwise it represents an element of an odometer).

%For every aperiodic point, the closure of it's orbit is a dyadic odometer.

\begin{figure}[ht]
\unitlength = 0.4cm
\begin{center}
\begin{graph}(28,14)
\graphnodesize{0.4}
% The top vertex
\roundnode{V0}(14,13)
% Vertices of the first level
\roundnode{V11}(6,9)
\roundnode{V12}(22,9)

% Edges of the first level
\edge{V11}{V0}
\bow{V12}{V0}{0.06}
\bow{V12}{V0}{-0.06}

%Vertices of the second level
\roundnode{V21}(2,4.5)
\roundnode{V22}(10,4.5)
\roundnode{V23}(18,4.5)
\roundnode{V24}(26,4.5)

% Edges of the second level
\edge{V21}{V11}
\bow{V22}{V11}{-0.06}
\bow{V22}{V11}{0.06}
\edge{V23}{V12}
\bow{V24}{V12}{-0.06}
\bow{V24}{V12}{0.06}

%Vertices of the third level
\roundnode{V31}(0,0.5)
\roundnode{V32}(4,0.5)
\roundnode{V33}(8,0.5)
\roundnode{V34}(12,0.5)
\roundnode{V35}(16,0.5)
\roundnode{V36}(20,0.5)
\roundnode{V37}(24,0.5)
\roundnode{V38}(28,0.5)
% Edges of the third level
\edge{V31}{V21}
\bow{V32}{V21}{-0.06}
\bow{V32}{V21}{0.06}
\edge{V33}{V22}
\bow{V34}{V22}{0.06}
\bow{V34}{V22}{-0.06}
\edge{V35}{V23}
\bow{V36}{V23}{0.06}
\bow{V36}{V23}{-0.06}
\edge{V37}{V24}
\bow{V38}{V24}{0.06}
\bow{V38}{V24}{-0.06}

\freetext(0,-0.5){$\vdots$}
\freetext(4,-0.5){$\vdots$}
\freetext(8,-0.5){$\vdots$}
\freetext(12,-0.5){$\vdots$}
\freetext(16,-0.5){$\vdots$}
\freetext(20,-0.5){$\vdots$}
\freetext(24,-0.5){$\vdots$}
\freetext(28,-0.5){$\vdots$}

\end{graph}
\end{center}

\vspace{0.5 cm}
\end{figure}

\end{example}

The next example comes from the survey~\cite{downarowicz_survey}. It was given there as an example of a non-decisive ordered Bratteli diagram, but in fact this diagram is weakly decisive.
For the reader's convenience, we recall the example below.

\begin{example}\label{medynets}(A non-decisive, weakly decisive, ordered Bratteli diagram: the Vershik map can be prolonged in many different ways, however, all prolongations are conjugate. Additionally, the cofinal equivalence relation is aperiodic, so, by Theorem \ref{main}, the system obtained by any prolongation must admit another, decisive diagram. This also follows from weak decisiveness and Theorem \ref{ble}). Consider the following diagram.

\begin{figure}[ht]
\unitlength = 0.4cm
\begin{center}
\begin{graph}(30,15)
\graphnodesize{0.4}
% The top vertex
\roundnode{V0}(15,14)
\freetext(15.5,14.6){$v_0$}
% Vertices of the first level
\roundnode{V11}(5,11)
\nodetext{V11}(-0.6,0.4){$u$}
\roundnode{V12}(15,11)
\nodetext{V12}(-0.6,0.4){$v$}
\roundnode{V13}(25,11)
\nodetext{V13}(0.6,0.4){$w$}

% Edges of the first level
\edge{V11}{V0}
\edgetext{V11}{V0}{0}
\edge{V12}{V0}
\edgetext{V12}{V0}{0}
\edge{V13}{V0}
\edgetext{V13}{V0}{0}

%Vertices of the second level
\roundnode{V21}(3,6)
\roundnode{V22}(7,6)
\roundnode{V23}(15,6)%odometer
\roundnode{V24}(23,6)
\roundnode{V25}(27,6)

%\freetext(19,6){$w$}

% Edges of the second level
\bow{V23}{V12}{-0.06}%odometer
\bow{V23}{V12}{0.06}%odometer
\freetext(14.3,8.5){0}
\freetext(15.7,8.5){1}

\edge{V21}{V11}
\edgetext{V21}{V11}{0}
\edge{V22}{V11}
\edgetext{V22}{V11}{0}
\edge{V21}{V12}
\edgetext{V21}{V12}{1}
\edge{V22}{V12}
\edgetext{V22}{V12}{1}

\edge{V24}{V12}
\edgetext{V24}{V12}{0}
\edge{V24}{V13}
\edgetext{V24}{V13}{1}
\edge{V25}{V12}
\edgetext{V25}{V12}{0}
\edge{V25}{V13}
\edgetext{V25}{V13}{1}
%\edgetext{V23}{V12}{0}

%Vertices of the third level

\roundnode{V31}(2,1)
\roundnode{V32}(4,1)
\roundnode{V33}(6,1)
\roundnode{V34}(8,1)

\roundnode{V35}(15,1)

\roundnode{V36}(22,1)
\roundnode{V37}(24,1)
\roundnode{V38}(26,1)
\roundnode{V39}(28,1)

\freetext(2,0){$\vdots$}
\freetext(4,0){$\vdots$}
\freetext(6,0){$\vdots$}
\freetext(8,0){$\vdots$}
\freetext(15,0){$\vdots$}
\freetext(22,0){$\vdots$}
\freetext(24,0){$\vdots$}
\freetext(26,0){$\vdots$}
\freetext(28,0){$\vdots$}

% Edges of the third level

\bow{V35}{V23}{-0.06}
\bow{V35}{V23}{0.06}
\freetext(14.3,3.5){0}
\freetext(15.7,3.5){1}

\edge{V31}{V21}
\edgetext{V31}{V21}{0}
\edge{V31}{V23}
\edgetext{V31}{V23}{1}

\edge{V32}{V21}
\edgetext{V32}{V21}{0}
\edge{V32}{V23}
\edgetext{V32}{V23}{1}

\edge{V33}{V22}
\edgetext{V33}{V22}{0}
\edge{V33}{V23}
\edgetext{V33}{V23}{1}

\edge{V34}{V22}
\edgetext{V34}{V22}{0}
\edge{V34}{V23}
\edgetext{V34}{V23}{1}

\edge{V36}{V24}
\edgetext{V36}{V24}{1}
\edge{V36}{V23}
\edgetext{V36}{V23}{0}

\edge{V37}{V24}
\edgetext{V37}{V24}{1}
\edge{V37}{V23}
\edgetext{V37}{V23}{0}

\edge{V38}{V25}
\edgetext{V38}{V25}{1}
\edge{V38}{V23}
\edgetext{V38}{V23}{0}

\edge{V39}{V25}
\edgetext{V39}{V25}{1}
\edge{V39}{V23}
\edgetext{V39}{V23}{0}

\end{graph}
\end{center}
\end{figure}

Any path in $X_B$ has infinitely many cofinal paths. Hence the Vershik map has infinite orbits and is aperiodic. Any path which passes through the vertex $u$ is minimal, while every path which passes through the vertex $w$ is maximal. These sets are clopen and homeomorphic to  the Cantor set. Additionally, there is one minimal and one maximal path, both passing through the vertex $v$ (they belong to the dyadic odometer represented by the central subdiagram and they do not belong to the interiors of $X_{\max}$ or $X_{\min}$). Thus, the sets of minimal and maximal paths are homeomorphic and have non-empty interiors, homeomorphic to the Cantor set. In order to be continuous, the prolongation of the Vershik map should map the maximal path passing through $v$ to the minimal path passing through $v$ and the remaining maximal paths (passing through $w$) to the minimal paths passing through $u$. Since the latter sets are both Cantor, there are multiple ways of homeomorphically mapping one to another, hence the diagram is not decisive. However, it is not hard to see that this diagram satisfies the conditions of weak decisiveness of Proposition~\ref{weak dec}; the paths passing through $w$, must be mapped to the paths passing through $u$. Then, after further $n$ steps all of them arrive to the central subdiagram at the level $\lfloor\log_2n\rfloor+1$ and from that place upward they agree with the $n$th image of the minimal path passing through $v$. So, the consecutive images of $X_{\max}$ are disjoint and their diameters shrink to zero.
\end{example}

\begin{example}\label{medynets1} (A non-weakly decisive diagram. Some prolongations are not conjugate. The cofinal relation is aperiodic, so by Theorem \ref{main}, any prolongation must admit another, decisive diagram.) This example resembles the previous one, except we ``double'' the central odometer subdiagram. The order of edges in the central diagrams is not shown because it is practically unique. As before, the interiors of $X_{\min}$ and of $X_{\max}$ consist of the paths passing through the vertices $u$ and $w$, respectively, are clopen in $X_B$ and homeomorphic to the Cantor set. There are also two more maximal and two more minimal paths which pass through the vertices $v_i$, for $i = 1,2$. Any continuous prolongation of the Vershik map must map the maximal path passing through the vertex $v_i$ to the minimal path passing through the same vertex $v_i$, for $i = 1,2$ but otherwise it can map $\mathsf{int} (X_{\max})$ onto $\mathsf{int} (X_{\min})$ by any homeomorphism between Cantor sets. This time, among these prolongations there are non-conjugate ones. For $i = 1,2$, let $X_{\min}^{(i)}$ be the set of all paths from $\mathsf{int} (X_{\min})$ whose forward orbits approach the odometer passing through the vertex $v_i$. Let $X_{\max}^{(i)}$ be the set of all paths from $\mathsf{int} (X_{\max})$ whose backward orbits approach the odometer passing through the vertex $v_i$. Let $h_1$ be a homeomorphism between $\inte(X_{\max})$ and $\inte(X_{\min})$ which sends $X_{\max}^{(i)}$ onto $X_{\min}^{(i)}$ and let $h_2$ send $X_{\max}^{(i)}$ to $X_{\min}^{(1 - i)}$. Let $T_i$ be the prolongation of the Vershik map with the help of the homeomorphism $h_i$. Then the system $(X_B, T_1)$ is a disjoint union of two closed subsystems, while $(X_B, T_2)$ does not split in this way. Thus the systems $(X_B, T_1)$ and $(X_B, T_2)$ are not conjugate.

\begin{figure}[ht]
\unitlength = 0.4cm
\begin{center}
\begin{graph}(30,15)
\graphnodesize{0.4}
% The top vertex
\roundnode{V0}(15,14)
\freetext(15.5,14.6){$v_0$}
% Vertices of the first level
\roundnode{V11}(3,11)
\nodetext{V11}(-0.6,0.4){$u$}
\roundnode{V12}(13,11)
\nodetext{V12}(-0.6,0.4){$v_1$}
\roundnode{V122}(17,11)
\nodetext{V122}(0.6,0.4){$v_2$}
\roundnode{V13}(27,11)
\nodetext{V13}(0.6,0.4){$w$}

% Edges of the first level
\edge{V11}{V0}
%\edgetext{V11}{V0}{0}
\edge{V12}{V0}
%\edgetext{V12}{V0}{0}
\edge{V122}{V0}
%\edgetext{V122}{V0}{0}
\edge{V13}{V0}
%\edgetext{V13}{V0}{0}

%Vertices of the second level
\roundnode{V21}(1,6)
\roundnode{V22}(5,6)
\roundnode{V23}(13,6)%odometer
\roundnode{V232}(17,6)%odometer
\roundnode{V24}(25,6)
\roundnode{V25}(29,6)

%\freetext(19,6){$w$}

% Edges of the second level
\bow{V23}{V12}{-0.06}%odometer
\bow{V23}{V12}{0.06}%odometer
\bow{V232}{V122}{-0.06}%odometer
\bow{V232}{V122}{0.06}%odometer
%\freetext(14.3,8.5){0}
%\freetext(15.7,8.5){1}

\edge{V21}{V11}
\edgetext{V21}{V11}{0}
\edge{V22}{V11}
\edgetext{V22}{V11}{0}
\edge{V21}{V12}
\edgetext{V21}{V12}{1}
\edge{V22}{V122}
\edgetext{V22}{V122}{1}

\edge{V24}{V12}
%\edgetext{V24}{V12}{0}
\edge{V24}{V13}
\edgetext{V24}{V13}{1}
\edge{V25}{V122}
\edgetext{V25}{V122}{0}
\edge{V25}{V13}
\edgetext{V25}{V13}{1}
\edgetext{V24}{V12}{0}

%Vertices of the third level

\roundnode{V31}(0,1)
\roundnode{V32}(2,1)
\roundnode{V33}(4,1)
\roundnode{V34}(6,1)

\roundnode{V35}(13,1)
\roundnode{V352}(17,1)

\roundnode{V36}(24,1)
\roundnode{V37}(26,1)
\roundnode{V38}(28,1)
\roundnode{V39}(30,1)

\freetext(0,0){$\vdots$}
\freetext(2,0){$\vdots$}
\freetext(4,0){$\vdots$}
\freetext(6,0){$\vdots$}
\freetext(13,0){$\vdots$}
\freetext(17,0){$\vdots$}
\freetext(24,0){$\vdots$}
\freetext(26,0){$\vdots$}
\freetext(28,0){$\vdots$}
\freetext(30,0){$\vdots$}

% Edges of the third level

\bow{V35}{V23}{-0.06}
\bow{V35}{V23}{0.06}
\bow{V352}{V232}{-0.06}
\bow{V352}{V232}{0.06}
%\freetext(14.3,3.5){0}
%\freetext(15.7,3.5){1}

\edge{V31}{V21}
\edgetext{V31}{V21}{0}
\edge{V31}{V23}
\edgetext{V31}{V23}{1}

\edge{V32}{V21}
\edgetext{V32}{V21}{0}
\edge{V32}{V23}
\edgetext{V32}{V23}{1}

\edge{V33}{V22}
\edgetext{V33}{V22}{0}
\edge{V33}{V232}
\edgetext{V33}{V232}{1}

\edge{V34}{V22}
\edgetext{V34}{V22}{0}
\edge{V34}{V232}
\edgetext{V34}{V232}{1}

\edge{V36}{V24}
\edgetext{V36}{V24}{1}
\edge{V36}{V23}
\edgetext{V36}{V23}{0}

\edge{V37}{V24}
\edgetext{V37}{V24}{1}
\edge{V37}{V23}
\edgetext{V37}{V23}{0}

\edge{V38}{V25}
\edgetext{V38}{V25}{1}
\edge{V38}{V232}
\edgetext{V38}{V232}{0}

\edge{V39}{V25}
\edgetext{V39}{V25}{1}
\edge{V39}{V232}
\edgetext{V39}{V232}{0}

\end{graph}
\end{center}
\end{figure}

\end{example}

\begin{example}
The diagram below gives an example of a non-Bratteli-Vershikizable system, such that the set of periodic has non-trivial boundary. The ordering is not shown since there is practically just one order. The set of fixed points is a Cantor set whose boundary is a one-point set (one of the fixed points lies in the closure of the set of aperiodic points). In the provided diagram, one can easily prolong the Vershik map to the interior of the set of periodic points in many different ways. But we know much more than that. Since the interior of the set of fixed points is non-empty, by Theorem~\ref{main}, the system is not Bratteli-Vershikizable, so that any BV-model (not only the one provided) is not decisive. Moreover, by Theorem \ref{ble}, the Vershik map on any BV-model can be prolonged in many mutually non-conjugate ways (this cannot be easily deduced by just looking at the provided diagram).

\begin{figure}[ht]
\unitlength = 0.4cm
\begin{center}
\begin{graph}(30,15)
\graphnodesize{0.4}
% The top vertex
\roundnode{V0}(15,14)
% Vertices of the first level
\roundnode{V11}(10,11)
\roundnode{V12}(15,11)
\roundnode{V13}(20,11)

% Edges of the first level
\edge{V11}{V0}
\edge{V12}{V0}
\bow{V13}{V0}{-0.06}
\bow{V13}{V0}{0.06}

%Vertices of the second level
\roundnode{V21}(5,7)
\roundnode{V22}(8,7)
\roundnode{V23}(12,7)
\roundnode{V24}(15,7)
\roundnode{V25}(20,7)
\roundnode{V26}(25,7)

% Edges of the second level
\edge{V12}{V23}
\edge{V11}{V22}
\edge{V11}{V21}
\edge{V12}{V24}
\bow{V25}{V12}{-0.06}
\bow{V25}{V12}{0.06}
\bow{V26}{V13}{-0.06}
\bow{V26}{V13}{0.06}

%Vertices of the third level

\roundnode{V31}(1,3)
\roundnode{V32}(3,3)
\roundnode{V33}(5,3)
\roundnode{V34}(7,3)

\roundnode{V35}(9,3)
\roundnode{V36}(11,3)
\roundnode{V37}(13,3)

\roundnode{V38}(15,3)
\roundnode{V39}(20,3)
\roundnode{V310}(25,3)
\roundnode{V311}(30,3)

% Edges of the third level

\edge{V21}{V31}
\edge{V21}{V32}
\edge{V22}{V33}
\edge{V22}{V34}
\edge{V23}{V35}
\edge{V37}{V24}
\edge{V36}{V23}
\edge{V24}{V38}

\bow{V24}{V39}{-0.06}
\bow{V24}{V39}{0.06}

\bow{V25}{V310}{-0.06}
\bow{V25}{V310}{0.06}

\bow{V26}{V311}{-0.06}
\bow{V26}{V311}{0.06}

\freetext(1,2){$\vdots$}
\freetext(3,2){$\vdots$}
\freetext(5,2){$\vdots$}
\freetext(7,2){$\vdots$}
\freetext(9,2){$\vdots$}
\freetext(11,2){$\vdots$}
\freetext(13,2){$\vdots$}
\freetext(15,2){$\vdots$}
\freetext(20,2){$\vdots$}
\freetext(25,2){$\vdots$}
\freetext(30,2){$\vdots$}

\end{graph}
\end{center}
\end{figure}

\end{example}

\section{A decisive Bratteli-Vershik model for the full shift}\label{fullshift}
In this section we show how to build an ordered decisive diagram for the full shift $\Omega$ on two symbols $\{0,1\}$. To follow the proof of Theorem~\ref{main}, we consider the array representation which, for every point $x \in \Omega$, consists of $x$ repeated in every row. In this array representation we will put markers satisfying the assertions of Lemma~\ref{marker_lemma}. 

We endow the set of all possible blocks of a given length $k$ with the lexicographical order: we say that a block $a$ \emph{dominates} a block $b$ if either $a=b$ or $a=c1x$ and $b=c0y$ for some blocks $c,x,y$ of appropriate lengths (possibly $0$). We put a marker at the horizontal position $n$ (in row $k$ of $x$) if (and only if) there exists $i \in [n-k+1,n]$ such that the block $x[n,n+k)$ dominates the blocks $x[i,i+k),x[i+1,i+k+1),\dots,x[i+k-1,i+2k-1)$ (roughly speaking, the block starting at $n$ must dominate all blocks starting in \emph{some} window of length $k$, containing $n$). Obviously, such a procedure of adding markers is continuous and shift-equivariant. 

Now we prove that the assertions of Lemma~\ref{marker_lemma} are satisfied. Note that if a block $a[0,k]$ dominates a block $b[0,k]$ then the subblock $a[0,k)$ dominates the subblock $b[0,k)$. This easily implies that the markers in row $k+1$ appear only at horizontal positions of the markers in row 
$k$ as required in condition (1) of Lemma~\ref{marker_lemma}.

Further, since among the blocks (of length $k$) which start in a window of length $k$ there is a dominating one, and it generates a marker, it is clear that markers in row $k$ appear at most $k$ positions apart as required in condition (2) of Lemma~\ref{marker_lemma}.

As we have already noted at the beginning of the proof of Lemma \ref{marker_lemma}, to prove condition (3) it suffices to show that the set of points which, for some $k$, do not have a marker at the coordinate $0$ in row $k$, is dense in $\Omega$. Consider points of the following form $x=\dots1111B01111\dots$, where $B$ is any finite block of an odd length $2l-1$ centered at the coordinate $0$. Clearly, such points form a set dense in $\Omega$. It is easy to see that if $k$ is much larger than $l$ then the block $x[0,k)$ (which has a zero already at the $l$th position) dominates neither of the following two blocks: $x[l+1,l+1+k)$ (which consists only of $1$'s) and $x[-k+l+1,l+1)$ (which starts with more than $l$ symbols $1$). Since in any window of length $k$, containing $n=0$, at least one of these two blocks starts, the algorithm does not produce a marker at the coordinate $0$.

We will now draw the first three levels of the diagram which corresponds to this markered array representation. We begin by listing all possible $k$-trapezoids for $k=1,2,3$. We will make two modifications. First, we will put the markers (vertical bars) to the left of the symbols rather than to the right, this will allow us to see the dominating block (which generates the marker) to the right of the vertical bar (of course, in most cases, this block is cut by the following marker, but in some cases we can see it completely between two markers). Second, we define the $k$-trapezoids slightly differently: we select a subsequence of rows, say $(k_i)$ and we widen the $k$-rectangles by adding $k_i$-rectangles on either sides only for $k_i<k$ (subsequently with decreasing $i$). In our particular example, since we draw only the levels $0$ through $3$, we have assumed that $k_1=2$ and $k_2>3$, so that each $2$-trapezoid consists of a $2$-rectangle widened in row $1$ by one $1$-rectangle on either side (and $1$-rectangles have length $1$), but each $3$-trapezoid consists of a $3$-rectangle widened still only in row $1$, by one symbol on either side. This reduces the cardinality of $3$-trapezoids from more than $50$ (if the trapezoids were defined as in the main proof) to $15$, a number reasonable to be drawn.

Since there are markers at all horizontal positions in row 1, the set $V_1$ consists of two vertices: 
%\begin{figure}[ht]

\[
\begin{array}{c}
\vline\,0\,\vline
\end{array}
\begin{array}{c}
\vline\,1\,\vline
\end{array}
\]

%\end{figure}
\medskip

The set $V_2$ consists of 11 vertices, which correspond to the following trapezoids:

\begin{figure}[ht]

\[
\begin{array}{c}
0\,\vline\,0\,\vline\,0\\
\vline\,0\,\vline
\end{array}
\begin{array}{c}
0\,\vline\,0\,\vline\,1\\
\vline\,0\,\vline
\end{array}
\begin{array}{c}
1\,\vline\,0\,\vline\,0\\
\vline\,0\,\vline
\end{array}
\begin{array}{c}
0\,\vline\,1\,\vline\,0\,\vline\,1\\
\vline\,1\phantom{\,\vline\,}0\,\vline
\end{array}
\begin{array}{c}
0\,\vline\,1\,\vline\,0\,\vline\,0\\
\vline\,1\phantom{\,\vline\,}0\,\vline
\end{array}
\begin{array}{c}
1\,\vline\,1\,\vline\,0\,\vline\,0\\
\vline\,1\phantom{\,\vline\,}0\,\vline
\end{array}
\begin{array}{c}
1\,\vline\,1\,\vline\,0\,\vline\,1\\
\vline\,1\phantom{\,\vline\,}0\,\vline
\end{array}
\begin{array}{c}
0\,\vline\,1\,\vline\,0\\
\vline\,1\,\vline
\end{array}
\begin{array}{c}
0\,\vline\,1\,\vline\,1\\
\vline\,1\,\vline
\end{array}
\begin{array}{c}
1\,\vline\,1\,\vline\,0\\
\vline\,1\,\vline
\end{array}
\begin{array}{c}
1\,\vline\,1\,\vline\,1\\
\vline\,1\,\vline
\end{array}
\]

\end{figure}

The set $V_3$ consists of 15 vertices, corresponding to the following trapezoids:

\begin{figure}[ht]

\[
\begin{array}{c}
0\,\vline\,0\,\vline\,0\\
\vline\,0\,\vline\\
\vline\,0\,\vline
\end{array}
\begin{array}{c}
0\,\vline\,0\,\vline\,1\\
\vline\,0\,\vline\\
\vline\,0\,\vline
\end{array}
\begin{array}{c}
1\,\vline\,0\,\vline\,0\\
\vline\,0\,\vline\\
\vline\,0\,\vline
\end{array}
\begin{array}{c}
0\,\vline\,1\,\vline\,0\,\vline\,1\\
\vline\,1\phantom{\,\vline\,}0\,\vline\\
\vline\,1\phantom{\,\vline\,}0\,\vline
\end{array}
\begin{array}{c}
0\,\vline\,1\,\vline\,0\,\vline\,0\,\vline\,0\\
\vline\,1\,\vline\,0\,\vline\,0\,\vline\\
\vline\,1\phantom{\,\vline\,}0\phantom{\,\vline\,}0\,\vline
\end{array}
\begin{array}{c}
0\,\vline\,1\,\vline\,0\,\vline\,0\,\vline\,1\\
\vline\,1\phantom{\,\vline\,}0\,\vline\,0\,\vline\\
\vline\,1\phantom{\,\vline\,}0\phantom{\,\vline\,}0\,\vline
\end{array}
\begin{array}{c}
1\,\vline\,1\,\vline\,0\,\vline\,0\,\vline\,0\\
\vline\,1\,\vline\,0\,\vline\,0\,\vline\\
\vline\,1\phantom{\,\vline\,}0\phantom{\,\vline\,}0\,\vline
\end{array}
\]

\end{figure}

\begin{figure}[ht]

\[
\begin{array}{c}
1\,\vline\,1\,\vline\,0\,\vline\,0\,\vline\,1\\
\vline\,1\phantom{\,\vline\,}0\,\vline\,0\,\vline\\
\vline\,1\phantom{\,\vline\,}0\phantom{\,\vline\,}0\,\vline
\end{array}
\begin{array}{c}
0\,\vline\,1\,\vline\,1\,\vline\,0\,\vline\,1\\
\vline\,1\,\vline\,1\phantom{\,\vline\,}0\,\vline\\
\vline\,1\phantom{\,\vline\,}1\phantom{\,\vline\,}0\,\vline
\end{array}
\begin{array}{c}
1\,\vline\,1\,\vline\,1\,\vline\,0\,\vline\,1\\
\vline\,1\,\vline\,1\phantom{\,\vline\,}0\,\vline\\
\vline\,1\phantom{\,\vline\,}1\phantom{\,\vline\,}0\,\vline
\end{array}
\begin{array}{c}
1\,\vline\,1\,\vline\,0\,\vline\,1\\
\vline\,1\phantom{\,\vline\,}0\,\vline\\
\vline\,1\phantom{\,\vline\,}0\,\vline
\end{array}
\begin{array}{c}
0\,\vline\,1\,\vline\,0\\
\vline\,1\,\vline\\
\vline\,1\,\vline
\end{array}
\begin{array}{c}
0\,\vline\,1\,\vline\,1\\
\vline\,1\,\vline\\
\vline\,1\,\vline
\end{array}
\begin{array}{c}
1\,\vline\,1\,\vline\,0\\
\vline\,1\,\vline\\
\vline\,1\,\vline
\end{array}
\begin{array}{c}
1\,\vline\,1\,\vline\,1\\
\vline\,1\,\vline\\
\vline\,1\,\vline
\end{array}
\]

\end{figure}

The first three levels of the diagram are shown on the figure below (the edges with a common source are always ordered from right to left).

\begin{figure}[ht]
\unitlength = 0.4cm
\begin{center}
\begin{graph}(28,14)
\graphnodesize{0.4}
% The top vertex
\roundnode{V0}(15,13)
% Vertices of the first level
\roundnode{V11}(10,9)
\roundnode{V12}(20,9)

\roundnode{V21}(3.5,5)
\roundnode{V22}(5.5,5)
\roundnode{V23}(7.5,5)
\roundnode{V24}(9.5,5)
\roundnode{V25}(12.5,5)
\roundnode{V26}(15.5,5)
\roundnode{V27}(18.5,5)
\roundnode{V28}(20.5,5)
\roundnode{V29}(22.5,5)
\roundnode{V210}(24.5,5)
\roundnode{V211}(26.5,5)

\roundnode{V31}(0,1)
\roundnode{V32}(2,1)
\roundnode{V33}(4,1)
\roundnode{V34}(6,1)
\roundnode{V35}(8,1)
\roundnode{V36}(10,1)
\roundnode{V37}(16,1)%change
\roundnode{V38}(12,1)
\roundnode{V39}(14,1)
\roundnode{V310}(18,1)
\roundnode{V311}(20,1)
\roundnode{V312}(22,1)
\roundnode{V313}(24,1)
\roundnode{V314}(26,1)
\roundnode{V315}(28,1)

\freetext(0,0){$\vdots$}
\freetext(2,0){$\vdots$}
\freetext(4,0){$\vdots$}
\freetext(6,0){$\vdots$}
\freetext(8,0){$\vdots$}
\freetext(10,0){$\vdots$}
\freetext(12,0){$\vdots$}
\freetext(14,0){$\vdots$}
\freetext(16,0){$\vdots$}
\freetext(18,0){$\vdots$}
\freetext(20,0){$\vdots$}
\freetext(22,0){$\vdots$}
\freetext(24,0){$\vdots$}
\freetext(26,0){$\vdots$}
\freetext(28,0){$\vdots$}

\edge{V11}{V0}
\edge{V12}{V0}

\edge{V21}{V11}
\edge{V22}{V11}
\edge{V23}{V11}
\edge{V24}{V11}
\edge{V25}{V11}
\edge{V26}{V11}
\edge{V27}{V11}
\edge{V24}{V12}
\edge{V25}{V12}
\edge{V26}{V12}
\edge{V27}{V12}
\edge{V28}{V12}
\edge{V29}{V12}
\edge{V210}{V12}
\edge{V211}{V12}

\edge{V31}{V21}
\edge{V32}{V22}
\edge{V33}{V23}
\edge{V34}{V24}
\edge{V35}{V28}
\edge{V35}{V23}
\edge{V35}{V21}
\edge{V36}{V25}
\edge{V36}{V22}
\edge{V37}{V29}
\edge{V37}{V27}
\edge{V38}{V210}
\edge{V38}{V23}
\edge{V38}{V21}
\edge{V39}{V26}
\edge{V39}{V22}
\edge{V310}{V211}
\edge{V310}{V27}
\edge{V311}{V27}
\edge{V312}{V28}
\edge{V313}{V29}
\edge{V314}{V210}
\edge{V315}{V211}

\end{graph}
\end{center}
\end{figure}
As we can see, the diagram is fairly irregular and does not show any symmetries (no permutation of vertices yields a symmetric diagram). We do not know whether a different assignment of markers produces a nicer diagram. The one we propose seems to be the most natural.  

\textbf{Acknowledgement.}
The research of both authors is supported by the NCN (National Science Center, Poland) Grant 2013/08/A/ST1/00275

\end{document}